\documentclass[a4paper,11pt,twoside,reqno]{amsart} 
\DeclareSymbolFontAlphabet{\mathcal}{symbols}
\usepackage{amssymb}
\usepackage{graphicx}
\usepackage[cmtip,all]{xy}
\usepackage{enumerate}
\usepackage{color}
\usepackage[a4paper]{geometry}
\title[]{Singular decompositions of a cap product} 

\date{\today}

\author{David Chataur}
\address{Lafma\\
Universit\'e de Picardie Jules Verne\\
33, rue Saint-Leu\\
80039 Amiens Cedex~1\\
         France}
\email{David.Chataur@u-picardie.fr}

\author{Martintxo Saralegi-Aranguren}
\address{Laboratoire de Math{\'e}matiques de Lens\\  
      EA 2462 \\
      Universit\'e d'Artois\\
         SP18, rue Jean Souvraz\\
          62307 Lens Cedex\\
         France}
\email{martin.saraleguiaranguren@univ-artois.fr}

\author{Daniel Tanr\'e}
\address{D\'epartement de Math{\'e}matiques\\
         UMR 8524 \\
         Universit\'e de Lille~1\\
         59655 Villeneuve d'Ascq Cedex\\
         France}
\email{Daniel.Tanre@univ-lille1.fr}
\thanks{
This research was supported through the program ``Research in Pairs'' at the Mathematisches Forschungsinstitut Oberwolfach in 2016. The authors thank the MFO for its generosity and hospitality.
The third author was also supported by the MINECO grant MTM2013-41768-P and the ANR-11-LABX-0007-01  ``CEMPI''}

\keywords{Intersection homology; cap product; Poincar\'e duality}

\subjclass[2010]{55N33, 57P10, 57N80}

\theoremstyle{plain}
\newtheorem{theorem}{Theorem}

\newtheorem{proposition}{Proposition}[section]

\theoremstyle{definition}
\newtheorem{definition}[proposition]{Definition}
\newtheorem{example}[proposition]{Example}

\theoremstyle{remark}
\newtheorem{remark}[proposition]{Remark}

\newcommand{\secref}[1]{Section~\ref{#1}}

\newcommand{\thmref}[1]{Theorem~\ref{#1}}
\newcommand{\propref}[1]{Proposition~\ref{#1}}

\newcommand{\remref}[1]{Remark~\ref{#1}}

\newcommand{\defref}[1]{Definition~\ref{#1}}

\def\R{{\mathbb R}}

\def\ov{\overline}

\def\cM{{\mathcal M}}

\def\1{{\mathbf 1}}

\def\tc{{\mathtt c}}

\def\tv{{\mathtt v}}

\def\N{\mathbb{N}}

\def\R{\mathbb{R}}
\def\Z{\mathbb{Z}}

\def\id{{\rm id}}

\def\tN{{\widetilde{N}}}

\def\GM{{\rm GM}}
\def\TW{{\rm TW}}

\def\rc{{\mathring{\tc}}}
\def\tDelta{{\widetilde{\Delta}}}
\def\tcap{\,{\widetilde{\cap}\;}}
\def\fa{{\vartriangleleft}}
\begin{document}

 \begin{abstract}
 In the case of a compact orientable pseudomanifold, 
 a well-known theorem of M.~Gor\-esky and R.~MacPherson
 says that the cap product with a fundamental class factorizes through
the intersection homology groups. 
 In this work, we show that this classical cap product is compatible with a cap product 
 in intersection (co)-homology, that we have previously 
 introduced. 
 If the pseudomanifold is also normal,  for any commutative ring of coefficients, the existence of a classical
 Poincar\'e duality isomorphism is equivalent to the existence of an isomorphism between the intersection homology groups
 corresponding to the zero and the top perversities. 
  \end{abstract}
 
  \maketitle

Let $X$ be a compact oriented pseudomanifold and $[X]\in H_{n}(X;\Z)$ be its fundamental class.
In \cite{MR572580}, M. Goresky and R. MacPherson prove that the 
\emph{Poincar\'e duality map} defined by the cap product
$-\cap [X]\colon H^{k}(X;\Z)\to H_{n-k}(X;\Z)$
can be factorized as
\begin{equation}\label{equa:GM}
H^{k}(X;\Z)\xrightarrow{\alpha^{\ov{p}}}
H_{n-k}^{\ov{p}}(X;\Z)
\xrightarrow{\beta^{\ov{p}}} H_{n-k}(X;\Z),
\end{equation}
where the groups $H_{i}^{\ov{p}}(X;\Z)$ are the intersection homology groups for the perversity $\ov{p}$.
The study of the Poincar\'e duality map via a filtration on homology classes is also considered in
the thesis of C.~McCrory \cite{MR2622359}, \cite{MR530047}, using a Zeeman's spectral sequence.

In \cite[Section 8.1.6]{FriedmanBook}, G. Friedman asks for a factorization of the Poincar\'e duality map
through a cap product defined in the context of an intersection cohomology
recalled in \secref{sec:examples}.
He proves it with a restriction on the torsion part of the intersection cohomology.
In this work we answer positively, without restriction and for any commutative ring $R$ of coefficients, in the context of a cohomology $H^*_{\TW,\ov{p}}(-)$ 
obtained via a simplicial blow-up 
with an intersection cap product,
$-\cap [X]\colon H^k_{\TW,\ov{p}}(X;R)\xrightarrow{\cong} H_{n-k}^{\ov{p}}(X;R)$,
 defined in \cite[Section 11]{CST4}  and recalled in \secref{sec:cap}. 
Roughly, our main result consists in the fact that this ``intersection cap product'' corresponds 
to the ``classical cap product''. 
This property can be expressed as the commutativity of the next diagram.

\begin{theorem}\label{thm:factorisation}
Let $X$ be a  compact oriented $n$-dimensional pseudomanifold. For any perversity $\ov{p}$, 
 there exists a commutative diagram,
\begin{equation}\label{equa:factorisation}
 \xymatrix{
H^k(X;R) \ar[rr]^-{-\cap [X]}  \ar[d]_{\cM_{\ov{p}}^*}
\ar[rrd]^-{\alpha^{\ov{p}}}
&&
H_{n-k}(X;R)
\\
H^k_{\TW,\ov{p}}(X;R)
\ar[rr]^-{-\cap [X]}_{\cong}
&&
H_{n-k}^{\ov{p}}(X;R).
\ar[u]_{\beta^{\ov{p}}}
}
\end{equation}
\end{theorem}

In \cite{MR572580}, the spaces and maps  of (\ref{equa:GM}) appear in the piecewise linear setting. 
In the previous statement we are working with singular
homology and cohomology. However, we keep the same letter $\alpha^{\ov{p}}$.
The morphism $\beta^{\ov{p}}$ is generated by  the inclusion of the corresponding complexes
and the morphism
$\cM^*_{\ov{p}}$ is defined in \secref{sec:proofs}.

For a normal compact oriented $n$-dimensional pseudomanifold,
the specification of the previous statement to the constant perversity with value~0
gives a commutative diagram,
where $\ov t$ is the top perversity defined by $\ov{t}(i)=i-2$.
\begin{equation}\label{equa:capandcap}
 \xymatrix{
H^k(X;R) \ar[rr]^-{-\cap [X]}  \ar[dd]_{\cM_{\ov{0}}^*}^-{\cong}
&&
H_{n-k}(X;R)
\\
&&
H^{\ov{t}}_{n-k}(X;R)\ar[u]^{\cong}_{\beta^{\ov{t}}}\\
H^k_{\TW,\ov{0}}(X;R)
\ar[rr]^-{-\cap [X]}_{\cong}
&&
H_{n-k}^{\ov{0}}(X;R).
\ar[u]_{\beta^{\ov{0},\ov{t}}}
}
\end{equation}
As a consequence, we have the next characterization, expressed in  Goresky-MacPherson intersection homology. 
This extends to any commutative ring of coefficients the criterion established in  \cite{MR572580} and \cite{FriedmanBook} .

\begin{theorem}\label{thm:0t}
Let $X$ be a normal compact oriented $n$-dimensional pseudomanifold.
Then the following conditions are equivalent.
\begin{enumerate}[(i)]
\item  The Poincar\'eŽ duality map
$-\cap [X]\colon H^{k}(X;R)\to H_{n-k}(X;R)$
 is an isomorphism.
\item The natural map $\beta^{\ov{0},\ov{t}}\colon H_{n-k}^{\ov{0}}(X;R)\to H_{n-k}^{\ov{t}}(X;R)$,
induced by the canonical inclusion of the corresponding complexes,
is an isomorphism.
\end{enumerate}
\end{theorem}

We have chosen the setting of the original perversities of 
M. Goresky and R. MacPherson \cite{MR572580}.
However, the previous results remain true in more general situations.

In the last section, we quote the existence of a cup product structure on $H^*_{\TW,\ov{\bullet}}(-)$ and 
detail how it combines with this factorization.
We are also looking for a factorization involving an intersection cohomology defined from the dual of intersection
chains (see \cite{zbMATH06243610}) and denoted $H^*_{\GM,\ov{p}}(-)$.
In the case of a locally torsion free pseudomanifold (see \cite{MR699009}), in particular if $R$ is a field,
a factorization as in (\ref{equa:factorisation}) and (\ref{equa:capandcap}) has been established by 
G.~Friedman in \cite[Section 8.1.6]{FriedmanBook}
for that cohomology. In this book, G.~Friedman asks also for such factorization through a cap product without this restriction on torsion.
In the last section, we give an example showing that such factorization with $H^*_{\GM,\ov{p}}(-)$ may not exist.

\tableofcontents

The coefficients for homology and cohomology are taken in a commutative ring $R$ (with unity) and 
we do not mention it explicitly in the rest of this work.
The degree of an element $x$ of a graded module is denoted~$|x|$.
All the maps $\beta$, with subscript or superscript, are induced by canonical inclusions of complexes.
For any topological space, $Y$, we denote by
$\tc Y=Y\times [0,1]/Y\times\{0\}$
the cone on $Y$ and by
$\rc Y=Y\times [0,1[/Y\times\{0\}$
the open cone on $Y$.

\smallskip
We thank the referee for making suggestions which have contributed to improve the writing.

\section{Background on intersection homology and cohomology}\label{sec:recall}

We recall the basic definitions and properties we need, sending the reader to \cite{MR572580}, \cite{FriedmanBook},
\cite{CST1} or \cite{CST4} for more details.

\begin{definition}\label{def:pseudomanifold}
An $n$-dimensional \emph{pseudomanifold}  is a topological space, $X$, filtered by closed subsets,
$$
X_{-1}=\emptyset\subseteq X_0 \subseteq X_1 \subseteq \dots \subseteq X_{n-2} = X_{n-1} \subsetneqq X_n =X,
$$
such that, for any $i\in\{0,\dots,n\}$, 
$X_i\backslash X_{i-1}$ is an $i$-dimensional topological manifold or the empty set. Moreover, for each point
$x \in X_i \backslash X_{i-1}$, $i\neq n$, there exist
\begin{enumerate}[(i)]
\item an open neighborhood $V$ of $x$ in $X$, endowed with the induced filtration,
\item an open neighborhood $U$ of $x$ in  $X_i\backslash X_{i-1}$, 
\item a compact pseudomanifold $L$  of dimension $n-i-1$, whose cone $\rc L$ is endowed with the filtration
$(\rc L)_{i}=\rc L_{i-1}$,%
\item a homeomorphism, $\varphi \colon U \times \rc L\to V$, 
such that
\begin{enumerate}[(a)]
\item $\varphi(u,\tv)=u$, for any $u\in U$, where $\tv$  is the apex of the cone $\rc L$,
\item $\varphi(U\times \rc L_{j})=V\cap X_{i+j+1}$, for all $j\in \{0,\dots,n-i-1\}$.
\end{enumerate}
\end{enumerate}
The pseudomanifold $L$ is called a  \emph{link} of $x$. 
The pseudomanifold $X$ is called \emph{normal} if its links are connected.
\end{definition}

As in \cite{MR572580}, a \emph{perversity} is a map $\ov{p}\colon \N\to\Z$ such that  
$\ov{p}(0)=\ov{p}(1)=\ov{p}(2)=0$ and
$\ov{p}(i)\leq\ov{p}(i+1)\leq\ov{p}(i)+1$ for all $i\geq 2$. Among them, we quote
the null perversity $\ov{0}$ constant with value 0 and the \emph{top perversity} defined by
$\ov{t}(i)=i-2$. For any perversity $\ov{p}$, the perversity $D\ov{p}:=\ov{t}-\ov{p}$
is called the \emph{complementary perversity} of $\ov{p}$.

\medskip
In this work, we compute the intersection homology of a pseudomanifold $X$ via \emph{filtered simplices.} 
They are 
singular simplices $\sigma\colon \Delta\to X$ such that $\Delta$ admits a decomposition in join products,
$\Delta=\Delta_{0}\ast\dots\ast\Delta_{n}$ with
$\sigma^{-1}X_{i}=\Delta_{0}\ast\dots\ast\Delta_{i}$.
The \emph{perverse degree} of $\sigma$ is defined by
$\|\sigma\|=(\|\sigma\|_{0},\dots,\|\sigma\|_{n})$
with
$\|\sigma\|_{i}=\dim(\Delta_{0}\ast\dots\ast\Delta_{n-i})$.
A filtered simplex is called \emph{$\ov{p}$-allowable} if 
\begin{equation}\label{equa:permis}
\|\sigma\|_{i}\leq \dim\Delta-i+\ov{p}(i),
\end{equation}
for any $i\in\{0,\dots,n\}$. 
A singular (filtered) chain $\xi$ is $\ov{p}$-allowable if it can be written as a linear combination of $\ov{p}$-allowable filtered simplices,
and of \emph{$\ov{p}$-intersection} if $\xi$ and $\partial \xi$ are $\ov{p}$-allowable.
We denote by $C_{*}^{\ov{p}}(X)$ the complex of singular (filtered) chains of $\ov{p}$-intersection. 
In \cite[Th\'eor\`eme A]{CST3}, we have proved that $C_{*}^{\ov{p}}(X)$ is quasi-isomorphic to the 
singular intersection chain complex introduced by H. King in \cite{MR800845}.

Given an euclidean simplex $\Delta$, we denote by $N_{*}(\Delta)$ 
and  $N^*(\Delta)$ the associated simplicial chain and cochain complexes. 
For each face $F $ of $\Delta$, we write $\1_{F}$ the cochain of $N^*(\Delta)$ taking the value 1 on $F$ and 0 otherwise. 
If $F$ is a face of $\Delta$, we denote by $(F,0)$ the same face viewed as face of the cone $\tc\Delta=\Delta\ast [\tv]$ and by $(F,1)$ 
the face $\tc F$ of $\tc \Delta$. By extension, we use also the notation $(\emptyset,1)=\tc \emptyset =[\tv]$ for the 
apex. The corresponding cochains are denoted $\1_{(F,\varepsilon)}$ for $\varepsilon =0$ or $1$.

\medskip
A filtered simplex $\sigma\colon \Delta=\Delta_{0}\ast\dots\ast\Delta_{n}\to X$ is called \emph{regular}
 if $\Delta_{n}\neq \emptyset$. 
The cochain complex we use for cohomology is built on the blow-up's of regular filtered simplices. More precisely,
we set first
$$\tN^*_\sigma=\tN^*(\Delta)=N^*(\tc\Delta_0)\otimes\dots\otimes N^*(\tc\Delta_{n-1})\otimes N^*(\Delta_n).$$
With the previous convention, a basis of $\tN^*(\Delta)$ is composed of elements of the form
$\1_{(F,\varepsilon)}=\1_{(F_{0},\varepsilon_{0})}\otimes\dots\otimes \1_{(F_{n-1},\varepsilon_{n-1})}\otimes \1_{F_{n}}
\in\tN^*(\Delta)$, where 
$\varepsilon_{i}\in\{0,1\}$ and
$F_{i}$ is a face of $\Delta_{i}$ for $i\in\{1,\dots,n\}$ or the empty set with $\varepsilon_{i}=1$ if $i<n$.
We set
$|\1_{(F,\varepsilon)}|_{>s}=\sum_{i>s}(\dim F_{i}+\varepsilon_{i})$,
with the convention $\dim \emptyset=-1$.

\begin{definition}\label{def:degrepervers}
Let $\ell$ be an element of $\{1,\ldots,n\}$ and
$\1_{(F,\varepsilon)}
\in\tN^*(\Delta)$.
The  \emph{$\ell$-perverse degree} of 
$\1_{(F,\varepsilon)}\in N^*(\Delta)$ is
$$
\|\1_{(F,\varepsilon)}\|_{\ell}=\left\{
\begin{array}{ccl}
-\infty&\text{if}
&
\varepsilon_{n-\ell}=1,\\
|\1_{(F,\varepsilon)}|_{> n-\ell}
&\text{if}&
\varepsilon_{n-\ell}=0.
\end{array}\right.$$
In the general case of a cochain $\omega = \sum_b\lambda_b \ \1_{(F_b,\varepsilon_b) }\in\tN^*(\Delta)$ with 
$\lambda_{b}\neq 0$ for all $b$,
the \emph{$\ell$-perverse degree} is
$$\|\omega \|_{\ell}=\max_{b}\|\1_{(F_b,\varepsilon_b)}\|_{\ell}.$$
By convention, we set $\|0\|_{\ell}=-\infty$.
\end{definition}

If $\delta_{\ell}\colon \Delta' 
\to\Delta$ 
 is a face operator
 (i.e. an inclusion of a face of codimension~1)
  we denote by
$\partial_{\ell}\sigma$ the filtered simplex defined by
$\partial_{\ell}\sigma=\sigma\circ\delta_{\ell}\colon 
\Delta'\to X$.
The \emph{Thom-Whitney complex} of $X$ is the cochain complex $\tN^*(X)$ composed of the elements, $\omega$, 
associating to each regular filtered simplex
 $\sigma\colon \Delta_{0}\ast\dots\ast\Delta_{n}\to X$,
an element
 $\omega_{\sigma}\in \tN^*_{\sigma}$,  
such that $\delta_{\ell}^*(\omega_{\sigma})=\omega_{\partial_{\ell}\sigma}$,
for any face operator
 $\delta_{\ell}\colon\Delta'\to\Delta$
 with $\Delta'_{n}\neq\emptyset$. 
 (Here $\Delta'=\Delta'_{0}\ast\dots\ast\Delta'_{n}$ is the induced filtration.)
 The differential $d \omega$ is defined by
 $(d \omega)_{\sigma}=d(\omega_{\sigma})$.
 The  \emph{$\ell$-perverse degree} of $\omega\in \tN^*(X)$ is the supremum of all the $\|\omega_{\sigma}\|_{\ell}$ for all
 regular filtered simplices $\sigma\colon\Delta\to X$.
 
 A \emph{cochain $\omega\in\tN^*(X)$ is $\ov{p}$-allowable} if $\| \omega\|_{\ell}\leq \ov{p}(\ell)$ for any $\ell\in \{1,\dots,n\}$,
 and of $\ov{p}$-intersection if $\omega$ and $d\omega$ are $\ov{p}$-allowable. 
 We denote $\tN^*_{\ov{p}}(X)$ the complex of $\ov{p}$-intersection cochains and by 
 $H_{\TW,\ov{p}}^*({X})$ its homology called
 \emph{Thom-Whitney cohomology} (henceforth \emph{TW-cohomology)} of $X$ 
 for the perversity~$\ov{p}$. 
 
\section{Cap product and intersection homology}\label{sec:cap}

We first recall the definition and some basic properties of a cap product in intersection (co)-homology already
introduced in \cite[Section 11]{CST4}.

Let $\Delta=[e_{0},\dots,e_{m}]$ be an euclidean simplex. 
We denote by $[\Delta]$ its face of maximal dimension.
The classical cap product
 $$-\cap [\Delta]\colon N^*(\Delta)\to N_{m-*}(\Delta)$$
  is defined by
$$
\1_{F}\cap [\Delta]=\left\{
\begin{array}{cl}
[e_{r},\dots,e_{m}] 
&
\text{if } F=[e_{0},\dots,e_{r}], \text{ for any } r\in \{0,\dots,m\},\\
0
&
\text{otherwise.}
\end{array}\right.
$$
We extend it to filtered simplices $\Delta=\Delta_{0}\ast\dots\ast\Delta_{n}$
as follows. 

\medskip
Let $\tDelta=\tc\Delta_{0}\times\dots\times\tc\Delta_{n-1}\times\Delta_{n}$.
If $\1_{(F,\varepsilon)}=\1_{(F_{0},\varepsilon_{0})}\otimes\dots\otimes \1_{(F_{n-1},\varepsilon_{n-1})}\otimes \1_{F_{n}}
\in \tN^*(\Delta)$, we define:
\begin{eqnarray}
\1_{(F,\varepsilon)} \tcap \tDelta
&=&
(-1)^{\nu(F,\varepsilon,\Delta)}
(\1_{(F_{0},\varepsilon_{0})}\cap \tc[\Delta_{0}])\otimes\dots\otimes
(\1_{F_{n}}\cap [\Delta_{n}]),\nonumber\\
&\in&
\tN_{*}(\Delta):=N_{*}(\tc\Delta_{0})\otimes\dots\otimes N_{*}(\tc \Delta_{n-1})\otimes N_{*}(\Delta_{n}),\label{equa:lecapsurdelta}
\end{eqnarray}
where $\nu(F,\varepsilon,\Delta)=\sum_{j=0}^{n-1}(\dim\Delta_{j}+1) \,(\sum_{i=j+1}^n|(F_{i},\varepsilon_{i})|)
$, with the convention $\varepsilon_{n}=0$.

\medskip
We define now a morphism,
$\mu^{\Delta}_{*}\colon \tN_{*}(\Delta)\to N_{*}(\Delta)$, 
by describing it on the elements
$(F,\varepsilon)=(F_{0},\varepsilon_{0})\otimes\dots\otimes (F_{n-1},\varepsilon_{n-1})\otimes F_{n}$.
Let $\ell$ be the smallest integer, $j$, such that $\varepsilon_{j}=0$. We set
\begin{equation}\label{equa:eclatement}
\mu^{\Delta}_{*}(F,\varepsilon)=\left\{
\begin{array}{cl}
F_{0}\ast\dots\ast F_{\ell}&\text{if } 
\dim (F,\varepsilon) =\dim (F_{0}\ast\dots\ast F_{\ell}),\\
0&\text{otherwise.}
\end{array}\right.
\end{equation}
(Note that $\mu_{*}^{\Delta}(F,\varepsilon)\neq 0$ if and only if all pairs $(F_{i},\varepsilon_{i})$ are
either $(\text{vertex},0)$ or $(\emptyset,1)$
for any $i>\ell$.)
If this image is not equal to zero, the application $\mu_{*}^{\Delta}$ consists of a replacement of the tensor product
$(F_{0},\varepsilon_{0})\otimes \dots\otimes
(F_{n-1},\varepsilon_{n-1})\otimes F_{n}$
by the join
$F_{0}\ast\dots\ast F_{\ell}$.
Therefore, if we set
$\nabla=\Delta_{1}\ast\dots\ast\Delta_{n}$, the application
$\mu_{*}^{\Delta}$ can be decomposed as
\begin{equation}\label{equa:decomposelemu}
N_{*}(\tc \Delta_{0})\otimes \tN_{*}(\nabla)
\xrightarrow{\id \otimes \mu_{*}^{\nabla}}
N_{*}(\tc\Delta_{0})\otimes N_{*}(\nabla)
\xrightarrow{\mu_{*}^{\Delta_{0}\ast\nabla}}
N_{*}(\Delta).
\end{equation}
Moreover, the application $\mu^{\Delta}_{*}\colon \tN_{*}(\Delta)\to N_{*}(\Delta)$ is a chain map, 
\cite[Proposition 11.10]{CST4} which allows the next local and global definitions of the intersection cap product.

\begin{definition}\label{def:localglobalcap}
Let $\Delta=\Delta_{0}\ast\dots\ast\Delta_{n}$ be a regular filtered simplex of dimension~$m$. 
The \emph{intersection cap product}
$-\cap \tDelta\colon \tN^*(\Delta)\to N_{m-*}(\Delta)$ is defined by
$$\omega\cap \tDelta=\mu^{\Delta}_{*}(\omega\tcap\tDelta).$$
\end{definition}

In short, we have introduced three maps, called cap products,
\begin{itemize}
\item the classical one, 
$-\cap [\Delta]\colon N^*(\Delta)\to N_{m-*}(\Delta)$,
\item its extension at the blow-up level, 
$-\tcap \tDelta\colon \tN^*(\Delta)\to \tN_{m-*}(\Delta)$,
\item and finally the projection on chains,
$-\cap \tDelta\colon \tN^*(\Delta)\to N_{m-*}(\Delta)$,
which is our intersection cap product.
We note that the apexes of the cones do not appear in an intersection cap product
since it is an element of $N_{*}(\Delta)$.
\end{itemize}

\begin{definition}\label{def:globalcap}
Let $X$ be a pseudomanifold,  $\omega\in\tN^*(X)$ and  $\sigma\colon \Delta_{\sigma}\to X$ a filtered simplex.
\emph{The intersection cap product}
$-\cap -\colon \tN^k(X)\otimes C_{k+j}(X)\to C_{j}(X)$ is defined by the linear extension of
$$\omega\cap \sigma=\left\{
\begin{array}{cl}
\sigma_{*}(\omega_{\sigma}\cap \tDelta_{\sigma}) 
&
\text{if } \sigma \text{ is regular,}\\
0&
\text{otherwise.}
\end{array}\right.$$
\end{definition}

\begin{proposition}[{\cite[Proposition 11.16]{CST4}}]\label{prop:lecap}
Let $X$ be an $n$-dimensional pseudomanifold and let $\ov{p}$, $\ov{q}$ be two  perversities. 
The cap product defines a chain map,
$$-\cap - \colon \tN^k_{\ov{p}}(X)\otimes C_{k+j}^{\ov{q}}(X) \to C_{j}^{\ov{p}+\ov{q}}(X).$$
\end{proposition}

\begin{remark}\label{rem:capcone}
Let $i\leq n-1$ and suppose $\Delta_{i}\neq \emptyset$.
The cone vertex $\tv_{i}$ belongs to the face $(F_{i},1)$ of $\tc\Delta_{i}$ and,
by convention, we write it as the last summit of $(F_{i},1)$, i.e.
$(F_{i},1)=[a_{j_{0}},\dots,a_{j_{r}},\tv_{i}]$
with $F_{i}=[a_{j_{0}},\dots,a_{j_{r}}]$.
If $F_{i}\neq \emptyset$, we denote by $G_{i}$ the face of $\Delta_{i}$ such that the cup product satisfies
$\1_{F_{i}}\cup \1_{G_{i}}=\1_{\Delta_{i}}$. 
From the definition of the cap product we have
\begin{equation}\label{equa:cap1}
\1_{(F_{i},\varepsilon_{i})}\cap [\tc\Delta_{i}]=\left\{
\begin{array}{cl}
(G_{i},1)&\text{if } \varepsilon_{i}=0,\\
{[ \tv_{i} ]} & \text{if }\varepsilon_{i}=1 \text{ and } F_{i}=\Delta_{i},
\\
0&\text{otherwise.}
\end{array}\right.
\end{equation}
Similarly, if $\Delta_{i}=\emptyset$ we have $\1_{(\emptyset,1)}\cap [\tc \emptyset]=[\tv_{i}]$.
\end{remark}
\section{Proofs of  Theorems  \ref{thm:factorisation} and  \ref{thm:0t}}\label{sec:proofs}

To define the morphism $\cM_{\ov{p}}^*$ of (\ref{equa:factorisation}), 
we  consider the morphism dual of $\mu_{*}^{\Delta}$ denoted
$\mu^*_{\Delta}\colon N^*(\Delta)\to \tN^*(\Delta)$.
Let
$$\tN_{\ov{0}}^*(\Delta)=\left\{\omega\in\tN^*(\Delta)\mid
\|\omega\|_{\ell}\leq 0 \text{ and }
\|d \omega\|_{\ell}\leq 0
\text{ for all } \ell\in\{1,\dots,n\}
\right\}.$$
The next result is in the spirit
of a theorem of Verona, see \cite{MR0290375}.

\begin{proposition}\label{prop:lau}
Let $\Delta=\Delta_{0}\ast\dots\ast \Delta_{n}$ be a regular filtered simplex. Then the chain map,
$$
\mu^*_{\Delta} \colon  N^*(\Delta) \to    \tN^*_{\ov 0}({\Delta})\subset \tN^*(\Delta),
$$
is an isomorphism.
\end{proposition}

\begin{proof}
Let  $F = F_0 \ast\dots\ast F_s$ be a face of $\Delta$ with $s\leq n$ and $F_{s}\neq \emptyset$. 
By definition of $\mu^*_{\Delta}$, we have
\begin{equation}\label{equa:lemusurG}
\mu^*_{\Delta}(\1_{F}) = 
\sum_{(a_{j_{s+1}},\dots,a_{j_{n}})}
(-1)^{\nu(F)}
\1_{(F_{0},1)}\otimes\dots\otimes \1_{(F_{s-1},1)}\otimes \1_{(F_{s},0)}
\otimes \1_{[a_{j_{s+1}}]}
\otimes\dots\otimes
\1_{[a_{j_{n}}]},
\end{equation}
where the $a_{j_{i}}$'s run over the vertices of $\tc\Delta_{i}$ if $i\in\{s+1,\dots,n-1\}$  and $a_{j_{n}}$ over the vertices of $\Delta_{n}$.
The sign is defined by
\begin{equation}\label{equa:lesignemu}
\left(\1_{(F_{0},1)}\otimes\dots\otimes \1_{(F_{s-1},1)}\otimes\1_{(F_{s},0)}\right)
\left([\tc F_{0}]\otimes\dots\otimes [\tc F_{s-1}]\otimes [F_{s}]\right)=(-1)^{\nu(F)}.
\end{equation}
From \defref{def:degrepervers}, we observe 
$\|\mu^*_{\Delta}(\1_{F})\|_{\ell}\leq 0$
for any $\ell\in\{1,\dots,n\}$
and the injectivity of $\mu^*_{\Delta}$.

Consider now  a cochain $\omega\in \tN^*_{\ov{0}}(\Delta)$. We
have to prove that $\omega$ belongs to the image of $\mu^*_{\Delta}$. 
Since $\|\omega\|_\ell \leq 0$ for each $\ell \in\{1,\ldots, n\}$,
we may write $\omega$ as a sum of
$$\omega_{F}=\sum_{(a_{j_{s+1}},\dots,a_{j_{n}})}
\lambda^F_{(a_{j_{s+1}},\dots,a_{j_{n}})}
\1_{(F_{0},1)}\otimes\dots\otimes \1_{(F_{s-1},1)}\otimes \1_{(F_{s},0)}
\otimes \1_{[a_{j_{s+1}}]}
\otimes\dots\otimes
\1_{[a_{j_{n}}]}
$$
where $\lambda^F_{(a_{j_{s+1}},\dots,a_{j_{n}})}\in R$ and $F=F_{0}\ast\dots\ast F_{s}$
with $F_{s}\neq \emptyset$
runs over the faces of $\Delta$.
Since $\omega\in \tN^*_{\ov{0}}(\Delta)$, we have 
$\|d\omega\|_{\ell}\leq 0$ for any $\ell\in \{1,\dots,n\}$.
Let $F=F_{0}\ast\dots\ast F_{s}$ with $F_{s}\neq \emptyset$ being fixed.
Since $d\1_{[a_{i}]}$ is of degree~1,
for having $\|d\omega\|_{\ell}\leq 0$ for any $\ell\in \{1,\dots,n\}$, we must have
$$
\sum_{(a_{j_{s+1}},\dots,a_{j_{n}})}
\lambda^F_{(a_{j_{s+1}},\dots,a_{j_{n}})}
(-1)^{|F|+s+1}
\1_{(F_{0},1)}\otimes\dots\otimes 
\1_{(F_{s},0)}
\otimes 
d(\1_{[a_{j_{s+1}}]}
\otimes\dots\otimes
\1_{[a_{j_{n}}]})=0,
$$
which implies
$$
\sum_{(a_{j_{s+1}},\dots,a_{j_{n}})}
\lambda^F_{(a_{j_{s+1}},\dots,a_{j_{n}})}
d(\1_{[a_{j_{s+1}}]}
\otimes\dots\otimes
\1_{[a_{j_{n}}]})=0.
$$
As, up to a multiplicative constant, there exists only one cocycle in degree zero in this tensor product, 
all the coefficients are equal, i.e.
there exists $\lambda_{F}\in R$ such that
$$\lambda_{F}=\lambda^F_{(a_{j_{s+1}},\dots,a_{j_{n}})},$$
for any $(n-s)$-uple of vertices $(a_{j_{s+1}},\ldots,a_{j_{n}})$.
Therefore, we may write
$$\omega=\mu^*_{\Delta}\left(\sum_{F\fa \Delta}\lambda_{F}\1_{F}\right)$$
and $\omega$ is in the image of $\mu^*_{\Delta}$.
\end{proof}

If $\omega\in \tN^*(\Delta)$ is the image by $\mu^*_{\Delta}$ of a cochain
$c\in N^*(\Delta)$, the intersection cap product coincides with the usual one.

 \begin{proposition}\label{prop:capcap}
 Let $\Delta=\Delta_{0}\ast\dots\ast \Delta_{n}$ be a regular filtered simplex of dimension~$m$.
For each cochain $c \in N^*(\Delta)$, we have
$$
\mu^*_{\Delta}(c)\cap\tDelta = c \cap [\Delta],
$$
where $c\cap [\Delta]$ comes from the usual cap product, $-\cap-\colon N^{*}(\Delta)\otimes N_{m}(\Delta)\to N_{m-*}(\Delta)$.
\end{proposition}

\begin{proof}
The result is clear for $n=0$. By using (\ref{equa:decomposelemu}), it is sufficient to prove the result for $\Delta=\Delta_{0}\ast \Delta_{1}$.
Let $\1_{F_{0}\ast F_{1}}\in N^*(\Delta)$.
We use \remref{rem:capcone} in the next determinations.\\
-- We note $G_{1}$ the face of $\Delta_{1}$ such that the cup product
$\1_{F_{1}}\cup \1_{G_{1}}$ 
is equal to $\1_{\Delta_{1}}$.\\
-- The cap product with $\1_{(F_{0},1)}$ is not equal to zero only if $F_{0}=\Delta_{0}$. In this case, we
set $G_{0}=\emptyset$. If $\varepsilon_{0}=0$, we note $G_{0}$ the face of $\Delta_{0}$ such that 
$\1_{F_{0}}\cup \1_{G_{0}}= \1_{\Delta_{0}}$.

\medskip
We prove the statement by considering the various possibilities.

$\bullet$ Suppose $F_{1}\neq\emptyset$. As
$\left(\1_{(F_{0},1)}\otimes \1_{F_{1}}\right)
\left([\tc F_{0}]\otimes [F_{1}]\right)=
(-1)^{|\tc F_{0}|\,|F_{1}|},
$
we deduce from (\ref{equa:lemusurG}) and (\ref{equa:lesignemu})
$$\mu^*_{\Delta}\left(\1_{F_{0}\ast F_{1}}\right)=
(-1)^{|F_{1}|\,|\tc F_{0}|}
\1_{(F_{0},1)}\otimes \1_{F_{1}}$$
and 
\begin{eqnarray*}
\mu^*_{\Delta}(\1_{F_{0}\ast F_{1}})\cap (\tc\Delta_{0}\times \Delta_{1})
&=&
 \mu_*^{\Delta} (\mu_{\Delta}^* (\1_{F_0* F_1})   \tcap (\tc \Delta_0 \times  \Delta_1 ))
 \\[.1cm]
&=&
(-1)^{|F_1||\tc F_0|}  \mu_*^{\Delta} ( ( \1_{(F_0,1)}\otimes  \1_{F_1} )\tcap (\tc \Delta_0 \times  \Delta_1 ))\\
&=&
(-1)^{|F_1|(|\tc F_0|+|\tc \Delta_{0}|)}  \mu_*^{\Delta} 
( \1_{(F_0,1)}  \cap [\tc \Delta_0]) \otimes ( \1_{F_1} \cap [ \Delta_1] )\\
&=_{(1)}&
\left\{
\begin{array}{ll}
\mu_*^{\Delta}( [\tv_{0}] \otimes G_1 ) & \hbox{if $F_{0}=\Delta_0$, 
}\\
0&\hbox{otherwise},
\end{array}
\right.\\
&=&
\left\{
\begin{array}{ll}
G_1 & \hbox{if $F_{0}=\Delta_0$, 
}\\
0&\hbox{otherwise},
\end{array}
\right.\\
&=&
 \1_{F_0* F_1}  \cap [\Delta_0*\Delta_1].
\end{eqnarray*}

$\bullet$ If $F_1 = \emptyset$, we have
$\mu^*_{\Delta}\left(\1_{F_{0}\ast\emptyset}\right)=
\sum_{[a_{j_{1}}] \fa \Delta_1} \1_{(F_0,0)}\otimes \1_{[a_{j_{1}}]} 
$
and
\begin{eqnarray*}
\mu^*_{\Delta}(\1_{F_{0}})\cap (c\Delta_{0}\times \Delta_{1})
&=&
 \mu_*^{\Delta} (\mu_{\Delta}^*( \1_{F_0} )  \tcap (\tc \Delta_0 \times  \Delta_1 ))
 \\
&=&
 \mu_*^{\Delta} 
 \left( (
\sum_{[a_{j_{1}}] \fa \Delta_1} \1_{(F_0,0)}\otimes \1_{[a_{j_{1}}]}   ) 
\tcap (\tc \Delta_0 \times  \Delta_1 )
\right)\\
&=& 
 \mu_*^{\Delta} 
 \left( 
\sum_{[a_{j_{1}}]\fa \Delta_1} (\1_{(F_0,0)}  \cap [\tc \Delta_0]  )\otimes (\1_{[a_{j_{1}}]}  \cap [\Delta_1] ) 
\right)\\
&=_{(1)}&
\mu_*^{\Delta}(  (G_0,1)  \otimes [ \Delta_1] ) 
=
G_0*\Delta_1
=
\1_{F_0}  \cap [\Delta_0*\Delta_1].
\end{eqnarray*}
(The two equalities $=_{(1)}$ are consequences of \remref{rem:capcone}.)
\end{proof}

Let us recall that  $C^*(X)$ is the complex of filtered singular cochains with coefficients in $R$.

\begin{proposition}\label{prop:Cero}
Let $X$ be a normal compact pseudomanifold. Then, the operator 
$
\cM_{\ov{0}} \colon  C^*(X) \to  \tN_{\ov 0}^*(X),
$
defined by $\cM_{\ov{0}}(c)_\sigma = \mu^*_{\Delta}(\sigma^*(c))$
for any regular filtered simplex $\sigma\colon \Delta\to X$,
is a  chain map which induces an isomorphism 
$$ \cM^*_{\ov{0}}\colon H^*(X) \xrightarrow{\cong}
  H_{\TW,\ov{0}}^*(X).
$$
\end{proposition}
We denote by $ \cM_{\ov{p}}\colon C^*(X)\to \tN^*_{\ov{p}}(X)$
the composition of $\cM_{\ov{0}}$ with the canonical inclusion of complexes and by
$ \cM^*_{\ov{p}}\colon H^*(X)\to H^*_{\TW,\ov{p}}(X)$
the induced morphism.

\begin{proof}
The maps $\mu_{*}^{\Delta}$ being compatible with restrictions, the map $\cM_{\ov{0}}$ is well defined. Its compatibility
with the differentials is a consequence of \cite[Proposition 11.10]{CST4}
and its behaviour with perversities a consequence of \propref{prop:lau}. 
For proving the isomorphism, we use a method similar to an argument of  H.~King in  \cite{MR800845},
see also \cite[Section 5.1]{FriedmanBook}. We have to check the hypotheses of \cite[Proposition 8.1]{CST4}. 

$\bullet$  The first one is the existence of Mayer-Vietoris sequences. This is clear for $C^*(-)$ and
has been proved in \cite[ThŽ\'eor\`eme A]{CST4}
for $\tN^*_{\ov{0}}(-)$.

$\bullet$ The second and the fourth hypotheses are straightforward.

$\bullet$ The third one consists in the computation of the intersection cohomology of a cone.  
Let $L$ be a compact connected pseudomanifold. It is well known that
$H^0(\R^i\times \rc L)=R$ and $H^i(\R^i\times \rc L)=0$ if $i>0$.
From \cite[Propositions 6.1 and 7.1]{CST4}, we have
$$\oplus_{k\geq 0}H^k_{\TW,\ov{0}}(\R^i\times \rc L)=H^0_{\TW,\ov{0}}(L).$$
If $\omega\in \tN_{\ov{0}}^0(L)$ is a cocycle, the cochain $\omega_{\sigma}$ is constant for any regular simplex
$\sigma\colon \Delta\to L$. As $\sigma$ has a non-empty intersection with the regular part, the connectedness of 
$L\backslash \Sigma$ (see \cite[Lemma 2.6.3]{FriedmanBook})
implies $H^0_{\TW,\ov{0}}(L;R)=R$.
\end{proof}

\propref{prop:Cero} remains true for normal CS-sets, with the same proof. We do not introduce this notion here,
see \cite{MR0319207} for its definition.

\begin{proof}[Proof of \thmref{thm:factorisation}]
Let $\ov{p}$ be a perversity.  
We consider the following diagram,
\begin{equation}\label{equa:thmA}
\xymatrix{
\ar[rr]^{-\cap [X] }  H^*(X) \ar[d]_{\cM^*_{\ov{0}}}
\ar@/^-3pc/[dd]_-{\cM^*_{\ov{p}}}
&&  
H_{n-*}(X)\\
H^*_{\TW,\ov{0}}(X)  \ar[d]_{\alpha_{\ov{0},\ov{p}}}
&&
 H ^{\ov{t}}_{n-*}(X) 
 \ar[u]_{\beta^{\ov{t}}}
 \\
 H^*_{\TW,\ov{p}}(X) \ar[rr]^{-\cap [X]}_{\cong}
 &&
 H^{\ov{p}}_{n-*}(X) \ar[u]_{\beta^{\ov{p},\ov{t}}},
}
\end{equation}
 where $\alpha_{\ov{0},\ov{p}}$, $\beta^{\ov{p},\ov{t}}$ and $\beta^{\ov{t}}$
 are induced by the natural inclusions of complexes. 
The bottom isomorphism comes from \cite[Th\'eŽor\`eme D]{CST4}. 
It remains to check the commutativity of this diagram.
Consider a cochain $c \in C^*(X)$ and a regular simplex $\sigma \colon \Delta_{\sigma} \to X$. 
(The chain $[X]$ being  $\ov 0$-allowable, each simplex in its decomposition is regular.)
We have
\begin{eqnarray}
\cM^*_{\ov{0}}(c)\cap\sigma
&=_{(1)}&
\sigma_{*}\left(\cM^*_{\ov{0}}(c)_{\sigma}\cap\tDelta_{\sigma}\right)
=_{(2)}
\sigma_{*}\left(\mu^*_{\Delta_{\sigma}}({\sigma}^*(c))\cap\tDelta_{\sigma}\right)
\nonumber
\\
&=_{(3)}&
\sigma_*\left({\sigma^*(c)}\cap [\Delta_{\sigma}]\right)
=_{(4)}
c\cap\sigma,\label{equa:M}
\end{eqnarray}
where (1) is \defref{def:globalcap}, (2) comes from the definition of $\cM^*_{\ov{0}}$,
(3) comes from \propref{prop:capcap}.
 and (4) is a property of the classical cap product.
\end{proof}

\begin{proof}[Proof of \thmref{thm:0t}]
By specifying the diagram (\ref{equa:thmA}) to the case $\ov{p}=\ov{0}$, we get the next commutative diagram.
$$
\xymatrix{
\ar[rr]^{-\cap [X]} \ar[dd]_{\cM^*_{\ov{0}}}^\cong H^*(X)
&&
H_{n-*}(X)\\
&&
H^{\ov{t}}_{n-*}(X) \ar[u]_{\beta^{\ov{t}}}^{\cong}\\
H^*_{\TW,\ov{0}}(X)\ar[rr]^{-\cap [X]}_{\cong}
&&
\ar[u]_{\beta^{\ov{0},\ov{t}}} H^{\ov{0}}_{n-*}(X).
}
$$
From \propref{prop:Cero}, \cite[Proposition 5.5]{CST3} and \cite[Th\'eor\`eme D]{CST4}, we get that  $\cM^*_{\ov{0}}$, $\beta^{\ov{t}}$ and the bottom map are isomorphisms. This ends the proof.
 \end{proof}

\section{Remarks and comments}\label{sec:examples}

\paragraph{\bf Cup product and intersection product}
In \cite{CST1}, \cite{CST4}, we define from the local structure on the Euclidean simplices, a cup product in intersection 
cohomology, induced by a chain map
\begin{equation}\label{equa:cup}
\cup\colon \tN^{k_{1}}_{\TW,\ov{p}_{1}}(X)
\otimes
\tN^{k_{2}}_{\TW,\ov{p}_{2}}(X)
\to
\tN^{k_{1}+k_{2}}_{\TW,\ov{p}_{1}+\ov{p}_{2}}(X).
\end{equation}
If $X$ is a compact oriented $n$-dimensional pseudomanifold, the Poincar\'e duality induces 
(see \cite[Section VIII.13]{MR0415602}) an intersection product
on the intersection homology, defined by the commutativity of the next diagram.
\begin{equation}\label{equa:intersectionprod}
\xymatrix{
H^{k_{1}}_{\TW,\ov{p}_{1}}(X)
\otimes
H^{k_{2}}_{\TW,\ov{p}_{2}}
\ar[rr]^-{\cup}
\ar[d]_{-\cap [X]\otimes -\cap [X]}^-{\cong}
&&
H^{k_{1}+k_{2}}_{\TW,\ov{p}_{1}+\ov{p}_{2}}(X)
\ar[d]^{-\cap [X]}_{\cong}
\\
H^{\ov{p}_{1}}_{n-k_{1}}(X)
\otimes
H^{\ov{p}_{2}}_{n-k_{2}}(X)
\ar[rr]^-{\pitchfork}
&&
H^{\ov{p}_{1}+\ov{p}_{2}}_{n-k_{1}-k_{2}}(X).
}
\end{equation}
Let $[\omega],[\eta]\in H^*(X)$ and $\alpha^{\ov{p}}=(-\cap [X])\circ \cM^*_{\ov{p}}$.
From this definition, we deduce
\begin{equation}\label{equa:alfafork}
\alpha^{2\ov{p}}([\omega]\cup [\eta])=
\alpha^{\ov{p}}([\omega])\pitchfork\,\alpha^{\ov{p}}([\eta]),
\end{equation}
which is the analogue of the decomposition established in \cite{MR572580} in the PL case.

\medskip
\paragraph{\bf Lattice structure on the set of perversities}
In diagram (\ref{equa:factorisation}), the cap product
$-\cap [X]\colon H^*(X)\to H_{n-*}(X)$ is factorized as $\beta^{\ov{p}}\circ\alpha^{\ov{p}}$
where $\beta^{\ov{p}}$ comes from the inclusion of complexes and
the data $\ov{p}\mapsto \alpha^{\ov{p}}$ is compatible with the lattice structure on the set of perversities.
More precisely, if $\ov{p}_{1}\leq \ov{p}_{2}$, we have a commutative diagram,
$$
\xymatrix{
&H^*(X)\ar@{-->}[dd]_-{\cM^*_{\ov p_1}}
\ar@{.>}[dddl]_-{\cM^*_{\ov p_2}}
\ar@{.>}[dddr]^-{\alpha^{\ov{p}_{2}}}
\ar@{-->}[rrdd]^-{\alpha^{\ov{p}_{1}}}
&&\\
&&&\\
&H^*_{\TW,\ov{p}_{1}}(X)
\ar@{-->}[rr]^-{-\cap [X]}
\ar[ld]^-{\alpha_{\ov{p}_{1},\ov{p}_{2}}}
&&
H^{\ov{p}_{1}}_{n-*}(X)\ar[ld]^-{\beta^{\ov{p}_{1},\ov{p}_{2}}}
\\
H^*_{\TW,\ov{p}_{2}}(X)
\ar@{.>}[rr]_-{-\cap [X]}
&&
H^{\ov{p}_{2}}_{n-*}(X),
&}$$
with $\alpha^{\ov{p}_{2}}=\beta^{\ov{p}_{1},\ov{p}_{2}}\circ \alpha^{\ov{p}_{1}}$. 

\medskip
\paragraph{\bf A second intersection cohomology}
Alternately, as cohomology theory, we could  choose (see \cite{FriedmanBook}, \cite{zbMATH06243610}) 
the dual complex of the intersection chains instead of $\tN^*_{\ov{p}}(X)$. We denote
\begin{equation}\label{equa:GM2}
C^*_{\GM,\ov{p}}(X;R)=\hom(C_{*}^{\ov{p}}(X;R),R)
\text{ and } H^*_{\GM,\ov{p}}(X;R) \text{ its cohomology.}
\end{equation} 
In \cite[Th\'eor\`eme C]{CST4}, we have proved
$$H^*_{\TW,\ov{p}}(X;R)\cong H^*_{\GM,D\ov{p}}(X;R),$$
if $R$ is a field, or more generally
with an hypothesis on the torsion of the links, introduced in \cite{MR699009}. 
But, in the general case, these two cohomologies may differ. A natural question is the existence of a  factorization
of $\alpha^{\ov{p}}$ as above but in which the cohomology $H^*_{\GM,D\ov{p}}(-)$ is substituted to $H^*_{\TW,\ov{p}}(-)$. 
This question
arises in \cite[Section 8.1.6]{FriedmanBook} together with a nice development on this point,
that we are taking back partially in these lines.
The next example shows that such factorization does not occur if we ask for the compatibility with the lattice structure.

\begin{example}
Recall that a cap product in the setting of $C^*_{\GM,\ov{p}}(X;R)$ has been introduced by G. Friedman and
J. McClure (see \cite{zbMATH06243610}) when $R$ is a field. 
In the case of a commutative ring, its definition requires a condition on the torsion
(see \cite{FriedmanBook}) which is satisfied in the case of the top perversity.
This justifies the existence of the horizontal isomorphism in the diagram below.

Consider the $4$-dimensional  compact oriented  pseudomanifold $X = \Sigma \R{\mathbb P}^3$ with 
 $R=\Z$, $\ov{p}_{1}=\ov{0}$, $\ov{p}_{2}=\ov{1}$. The analog of the previous diagram can be written as
$$\xymatrix{
H^3(X)=\Z_{2}
\ar[d]_{\cM^3_{\GM,\ov{2}}}^{\cong}
\ar[rrd]^{\alpha^{\ov{0}}_{\GM}}
&\\
H^3_{\GM,\ov{2}}(X)=\Z_{2}
\ar[rr]^{-\cap [X]}_{\cong}
\ar[d]
&&
H^{\ov{0}}_{1}(X)=\Z_{2}
\ar[d]^{\beta^{\ov{0},\ov{1}}}_{\cong}
\\
H^3_{\GM,\ov{1}}(X)=0
&&
H^{\ov{1}}_{1}(X)=\Z_{2},
}$$
where $\cM^3_{\GM,\ov{2}}$ corresponds to the morphism $\omega^*_{\ov{2}}$ in 
\cite[Section 8.1.6]{FriedmanBook}.
Let $\alpha^{\ov{0}}_{\GM}=(-\cap [X])\circ \cM^3_{\GM,\ov{2}}$. Then the map 
$\alpha^{\ov{1}}_{\GM}= \beta^{\ov{0},\ov{1}}\circ \alpha^{\ov{0}}_{\GM}$ is an isomorphism which cannot
factorize through $H^3_{\GM,\ov{1}}(X)$.
\end{example}

\providecommand{\bysame}{\leavevmode\hbox to3em{\hrulefill}\thinspace}
\providecommand{\MR}{\relax\ifhmode\unskip\space\fi MR }
\providecommand{\MRhref}[2]{%
  \href{http://www.ams.org/mathscinet-getitem?mr=#1}{#2}
}
\providecommand{\href}[2]{#2}

\end{document}